\documentclass{article}

\author{Tristan Robert}
\date{\today}

\usepackage[utf8]{inputenc} 
\usepackage[T1]{fontenc}
\usepackage{lmodern}
\usepackage{amsmath,amssymb,amsfonts,amsthm}
\usepackage[nottoc, notlof, notlot]{tocbibind}
\usepackage{bbold} 
\usepackage{thmtools}
\usepackage{enumitem}

\newcommand*{\R}{\mathbb{R}} 
\newcommand*{\N}{\mathbb{N}} 
\newcommand*{\Z}{\mathbb{Z}} 
 
\newcommand*{\D}{\mathcal{D}}

\newcommand*{\E}{\mathbf{E}} 
\newcommand*{\El}{\mathbf{E}}
 
\newcommand*{\T}{\mathbb{T}} 

\newcommand*{\X}{\mathcal{X}} 

\newcommand*{\fl}{F} 
\newcommand*{\nl}{N} 
\newcommand*{\Fl}{\mathbf{F}} 
\newcommand*{\Nl}{\mathbf{N}} 
\newcommand*{\Bl}{\mathbf{B}}

\newcommand*{\lpart}[1]{\left[#1\right]}

\newcommand*{\F}{\mathcal{F}} 

\newcommand*{\B}{\mathbf{B}}

\newcommand*{\supp}{\ensuremath{\mathrm{supp}} }

\newcommand*{\norme}[2]{\ensuremath{\left|\left|#2\right|\right|_{#1}}}

\newcommand*{\infeg}{\leqslant}
\newcommand*{\supeg}{\geqslant}

\newcommand*{\drt}{\partial_t}
\newcommand*{\drx}{\partial_x}

\newcommand*{\dry}{\partial_y}

\newcommand*{\normL}[2]{\ensuremath{\left|\left|#2\right|\right|_{L^{#1}}}}

\newcommand*{\dt}{\mathrm{d} t}

\newcommand*{\dx}{\mathrm{d} x}
\newcommand*{\dy}{\mathrm{d} y}

\newcommand*{\dtau}{\mathrm{d} \tau}

\newcommand*{\e}{\mathrm{e}}

\newcommand*{\crochet}[1]{{\displaystyle \left\langle #1\right\rangle }}

\newcommand*{\et}{\ensuremath{\wedge}}
\newcommand*{\ou}{\ensuremath{\vee}}

\renewenvironment{proof}[1][]{~~\\ \textbf{Proof #1 : }\\}{\begin{flushright}
$\Box$
\end{flushright} }

\usepackage[english]{babel}
\newtheorem{theoreme}{Theorem}[section]
\newtheorem{proposition}[theoreme]{Proposition}
\newtheorem{corollaire}[theoreme]{Corollary}
\newtheorem{lemme}[theoreme]{Lemma}
\newtheorem{remarque}[theoreme]{Remark}

\makeatletter
\@addtoreset{equation}{section}

\makeatother

\title{On the Cauchy problem for the periodic fifth-order {KP-I} equation}
\date{}
\author{Tristan Robert\\ \emph{Université de Cergy-Pontoise}\\ \emph{Laboratoire AGM}\\ \emph{2 av. Adolphe Chauvin, 95302 Cergy-Pontoise Cedex, France}\\ \emph{tristan.robert@u-cergy.fr}}

\begin{document}
\maketitle
\begin{abstract}
The aim of this paper is to investigate the Cauchy problem for the periodic fifth order KP-I equation \[\drt u - \drx^5 u -\drx^{-1}\dry^2u + u\drx u = 0,~(t,x,y)\in\R\times\T^2\]
We prove global well-posedness for constant $x$ mean value initial data in the space $\E = \{u\in L^2,~\drx^2 u \in L^2,~\drx^{-1}\dry u \in L^2\}$ which is the natural energy space associated with this equation.
\end{abstract}
\section{Introduction}
The KP equations arised in \cite{KP1970} as fluid mechanics models for long, weakly nonlinear two-dimensional waves with a small dependence in the tranverse variable. The usual KP equations are
\begin{equation}\label{equation KP standard}
\drt u + \drx^3 u + \epsilon \drx^{-1}\dry u + u\drx u=0
\end{equation} 
where the coefficient $\epsilon$ depends on the surface tension. The KP-I equation corresponds to $\epsilon=-1$, and the KP-II equation to $\epsilon = 1$. The Cauchy problem for these equations has been extensively studied in the past twenty years. The KP-II equation is known to be locally well-posed in the scale-critical space $H^{-1/2,0}(\R^2)$ \cite{HadacHerrKoch}, and globally well-posed in $L^2(\R\times\T)$ \cite{MST2011} and $L^2(\T^2)$ \cite{bourgain1993kp}.

As for the KP-I equation, some ill-posedness results \cite{MST2002,Koch2008} have shown that this equation does not have a semilinear nature, in the sense that it cannot be treated via a perturbative method. Ionescu, Kenig and Tataru \cite{IonescuKenigTataru2008} thus developped the short-time Fourier restriction norm method to overcome the resonant low-high interactions responsible of the quasilinear behavior, therefore obtaining global well-posedness in the energy space on $\R^2$. The adaption \cite{Zhang2015} in the periodic setting revealed a logaritheoremeic divergence in the energy estimate due to a bad frequency interaction in the resonant set, establishing therefore a local well-posedness result in the Besov space $\mathbf{B}^1_{2,1}(\T^2)$ which is strictly larger than the natural energy space. To overcome this difficulty and recover a global well-posedness result in the energy space, one can look for a better dispersion effect by either removing the assumption of periodicity in one direction \cite{Article1}, or studying higher-order models.

To pursue this latter issue, we investigate the Cauchy problem for the periodic fifth-order KP-I equation
\begin{equation}\label{equation KP1-5}
\drt u - \drx^5 u -\drx^{-1}\dry^2u + u\drx u = 0,~(t,x,y)\in\R\times\T^2
\end{equation}
First, as noticed by Bourgain \cite{bourgain1993kp} in the context of the periodic KP-II equation, any (periodic in space) solution of (\ref{equation KP1-5}) has a constant (in $y$) $x$-mean value, i.e if $(m,n)\in\Z^2$ are the Fourier variables associated with $(x,y)\in\T^2$, then the Fourier coefficients of $u$ with respect to $(x,y)$ satisfy the extra condition
\begin{equation}\label{condition coef}
\widehat{u}(t,0,n) = 0\text{ for }n\in\Z\setminus\{0\}
\end{equation}
In particular, in $t=0$ we see that the initial data must satisfy (\ref{condition coef}). As in \cite{bourgain1993kdv,bourgain1993kp}, we will make the additional assumption that $\widehat{u_0}(0,0)=0$, which is not restrictive since for data $u_0$ with non zero constant $x$ mean value, we will just have to set $v_0:= u_0- \widehat{u_0}(0,0)$ which satisfies the above condition and the modified equation
\[\drt v - \drx^5 v + c\drx v -\drx^{-1}\dry^2v  + v\drx v  = 0\]
with $c=\widehat{u_0}(0,0)$. Our analysis of (\ref{equation KP1-5}) applies equally to the above modified equation, since the extra lower order term does not change the resonnant function (see its definition in (\ref{definition fonction resonnance}) below). 

Now, to work with initial data satisfying the constraint (\ref{condition coef}), we introduce the subspace of distributions
\[\mathcal{D}'_0(\T^2):=\left\{u_0\in\mathcal{D}'(\T^2),~\widehat{u_0}(0,n)=0~\forall n\in\Z\right\}\]
in which the operator $\drx^{-1}$ is well defined as
\[\drx^{-1}u_0(x,y) := \F^{-1}\left\{\frac{1}{im}\widehat{u_0}(m,n)\right\}\]
The equation (\ref{equation KP1-5}) has another interesting feature : it possesses some conservation laws. Indeed,  the mass
\begin{equation}\label{definition masse KP}
\mathcal{M}(u_0) := \int_{\T^2}u_0^2(x,y)\dx\dy
\end{equation} and the energy
\begin{equation}\label{definition energie KP}
\mathcal{E}(u_0) := \int_{\T^2}\left\{(\drx^2 u_0)^2(x,y) + (\drx^{-1}\dry u_0)^2(x,y) - \frac{1}{3} u_0^3(x,y)\right\}\dx\dy
\end{equation}
are conserved by the flow. Therefore, to obtain a global well-posedness result, it suffices to construct local solutions to (\ref{equation KP1-5}) and they will be automatically extended globally in time as soon as the above quantities are bounded.

In view of the precedent remarks, we will thus work in the energy space defined as
\begin{multline}\label{definition espace donnee}
\E(\T^2):=\left\{u_0\in \mathcal{D}'_0(\T^2)\cap L^2(\T^2),~\drx^2u_0\in L^2(\T^2),~\drx^{-1}\dry u_0\in L^2(\T^2)\right\}
\end{multline}
 endowed with the norm
 \[\norme{\E}{u_0}:=\left(\normL{2}{u_0}^2+\normL{2}{\drx^2 u_0}^2 + \normL{2}{\drx^{-1}\dry u_0}^2\right)^{1/2}\] 
For initial data in this space, the mass is clearly finite, and due to the anisotropic Sobolev inequality of Tom \cite[Lemma 2.5]{Tom} the energy is bounded as well.
 
The first results on the Cauchy problem for (\ref{equation KP1-5}) were obtained by I\'{o}rio and Nunes in the general setting of \cite{IorioNunes} where it has been shown to be locally well-posed for zero mean value initial data in the space $H^s(\T^2)$ for $s>2$ by adapting the general quasi-linear theory of Kato. This model has then been studied in the work of Saut and Tzvetkov \cite{SautTzvetkov1999,SautTzvetkov2000,SautTzvetkov2001}, where it has been proved that this equation is globally well-posed in the energy spaces $\E(\R^2)$ and $\E(\T\times\R)$ by using the standard Bourgain method. Li and Xiao \cite{LiXiao2008} have then pushed forward with this approach and got global well-posedness in $L^2(\R^2)$. However, a counter-example is built in \cite{SautTzvetkov2001} to show the failure of the bilinear estimate in the usual Bourgain spaces when $u$ is periodic in both variables, initiating thereafter a systematic study of such quasilinear behaviours in dispersive equations (see \cite{Tzvetkov2004ill} for a detailed presentation of this issue). This implies that another approach is needed. Using the refined energy method of \cite{KochTzvetkov2003BO}, Ionescu and Kenig \cite{IonescuKenig2007} proved global well-posedness in $\E(\R\times\T)$. Very lately, Guo, Huo and Fang \cite{Guo2017} proved local well-posedness in $H^{s,0}(\R^2)$ for $s\supeg -3/4$ and the initial-value problem (\ref{equation KP1-5}) for periodic initial data in the energy space remained open. In this note, we prove the following.
\begin{theoreme}\label{theoreme}
\begin{enumerate}[label=(\alph*)]
\item\label{theoreme partie 1} For any $u_0\in\E^{\infty}(\T^2)$, there exists a unique global smooth solution \[u=:\Phi^{\infty}(u_0)\in\mathcal{C}(\R,\E^{\infty}(\T^2))\] to (\ref{equation KP1-5}) and moreover, for any $T>0$ and $\sigma\supeg 2$ we have
\begin{equation}\label{estimation energie avec donnee}
\norme{L^{\infty}_T\E^{\sigma}}{\Phi^{\infty}(u_0)}\infeg C\left(T,\sigma,\norme{\E^{\sigma}}{u_0}\right)
\end{equation}
\item\label{theoreme partie 2} Take any $u_0\in\E(\T^2)$ and $T>0$, then there exists a unique solution $u$ to (\ref{equation KP1-5}) in the class
\begin{equation}\label{definition classe unicite}
\mathcal{C}\left([-T;T],\E\right)\cap\mathbf{F}(T)\cap\B(T)
\end{equation}
This defines a continuous flow $\Phi : \E\rightarrow \mathcal{C}(\R,\E)$ which leaves $\mathcal{M}$ and $\mathcal{E}$ invariants.
\end{enumerate}
\end{theoreme}
The functions spaces $\E^{\infty}$, $\mathbf{F}(T)$ and $\B(T)$ are defined in section~\ref{section espaces fonctions} below.

Now, in view of the above definition of the energy space, one may be surprised by the gap in regularity between the Cauchy theory in $\R^2$ \cite{Guo2017} and our well-posedness result. This is explained by the difficulty to evaluate accurately the measure of the reasonant set in the periodic setting. See remark~\ref{remarque gap} below for more details.

To prove Theorem~\ref{theoreme}, we will then use the method of \cite{IonescuKenigTataru2008} and prove the linear, bilinear and energy estimates in the spaces $\mathbf{F}, \mathbf{N}$ and $\mathbf{B}$.

Section~\ref{section espaces fonctions} introduces general functions spaces and their basic properties. We prove some dyadic estimates in section~\ref{section estimation dyadique} which we will use in sections~\ref{section estimation energie} and~\ref{section estimation bilineaire} to prove energy and  bilinear estimates respectively. The proof of Theorem~\ref{theoreme} is finally completed in section~\ref{section preuve}.
\paragraph*{Notations}
For positive reals $a$ and $b$, $a\lesssim b$ means that there exists a positive constant $c>0$ (independant of the various parameters) such that $a \infeg c\cdot b$. \\
The notation $a\sim b$ stands for $a\lesssim b$ and $b\lesssim a$.\\
For $x\in\R^d$ we set $\crochet{x}:=(1+|x|^2)^{1/2}$.\\
For a real $x$, we write $\lpart{x}$ to denote its integer part.\\
For a set $A\subset\R^d$, $\mathbb{1}_A$ is the characteristic function of $A$ and if $A$ is Lebesgue-measurable, $|A|$ means its measure. When $A\subset \Z$ is a finite set, its cardinal is denoted $\#A$.\\
For $M>0$ and $s\in\R$, $\lesssim M^{s-}$ means $\infeg C_{\varepsilon}M^{s-\varepsilon}$ for any choice of $\varepsilon>0$ small enough. We define similarly $M^{s+}$.\\
Let $(\tau,m,n)\in \R\times\Z^2$ denote the Fourier variables of $(t,x,y)\in\R\times\T^2$.
We define the unitary group \[U(t) = \e^{-t(\drx^5+\drx^{-1}\dry^2)}= \F_{xy}^{-1}\e^{-it\omega(m,n)}\F_{xy}\] where $\omega(m,n):= m^5+\frac{n^2}{m}$.\\
We note $M,K\in 2^{\N}$ the dyadic frequency decompositions of $|m|$ and $\crochet{\tau+\omega(m,n)}$.
We define then ${\displaystyle D_{M,K} := \left\{(\tau,m,n)\in \R\times\Z^2,~~|m| \sim M, \crochet{\tau+\omega(m,n)}\sim K \right\}}$ and ${\displaystyle D_{M,\infeg K} := \bigcup_{K'\infeg K} D_{M,K'}}$.\\
We note also ${\displaystyle I_M := \left\{5M/8\infeg |m|\infeg 8M/5\right\}}$ and ${\displaystyle I_{\infeg M} := \bigcup_{M'\infeg M}I_{M'}}$.\\
We use the notations $M_1\et M_2 := \min(M_1,M_2)$ and $M_1 \ou M_2 := \max(M_1,M_2)$.\\
For $M_1,M_2,M_3 \in \R_+^*$, $M_{\min}\infeg M_{med}\infeg M_{max}$ denotes the increasing rearrangement of $M_1,M_2,M_3$, i.e 
\begin{multline*}
M_{min} := M_1\et M_2 \et M_3,~~ M_{max} = M_1 \ou M_2 \ou M_3 \\ \text{ and } M_{med} = M_1+M_2+M_3 - M_{max} - M_{min}
\end{multline*}
We define now the Littlewood-Paley decomposition. Let $\chi\in\mathcal{C}^{\infty}_c(\R)$ with $0\infeg \chi \infeg 1$, $\supp \chi \subset [-8/5;8/5]$ and $\chi \equiv 1$ on $[-5/4;5/4]$.\\
For $K\in 2^{\N}$, we then define $\eta_1(x):=\chi(x)$ and $\eta_K(x) := \chi(x/K) - \chi(2x/K)$ if $K>1$, such that $\supp \eta_K \subset I_K$ and $\eta_K \equiv 1$ on $\{4/5 K\infeg |x|\infeg 5/4 K\}$ for $K>1$. Thus $\crochet{\tau+\omega(m,n)} \in \supp \eta_K \Rightarrow \crochet{\tau+\omega} \in I_K$ and $|\tau+\omega| \sim K$ for any $K\in 2^{\N}$.\\
When needed, we may use another decomposition $\widetilde{\chi}$, $\widetilde{\eta}$ with the same properties as $\chi$, $\eta$ and satisfying $\widetilde{\chi}\equiv 1$ on $\supp \chi$ and $\widetilde{\eta}\equiv 1$ on $\supp \eta$.\\
Finally, for $\kappa\in\R_+^*$, we note $\chi_{\kappa}(x) := \chi(x/\kappa)$.

We also define the Littlewood-Paley projectors associated with the sets $I_M$ :
\[P_M u := \F^{-1}\left(\mathbb{1}_{I_M}(m)\widehat{u}\right) \text{ and }P_{\infeg M} u := \sum_{M'\infeg M}P_M u = \F^{-1}\left(\mathbb{1}_{I_{\infeg M}}(m)\widehat{u}\right)\]
\section{Functions spaces and first properties}\label{section espaces fonctions}
\subsection{Definitions}
The energy space $\E$ was defined in (\ref{definition espace energie}). 
More generally, for $\sigma\supeg 2$, we define
\begin{multline*}\label{definition E sigma}
\E^{\sigma}(\T^2) := \left\{u_0\in\D_0'(\T^2)\cap L^2(\T^2),~ \norme{\E^{\sigma}}{u_0}:= \normL{2}{\crochet{m}^{\sigma}\cdot p(m,n)\cdot \widehat{u_0}}<+\infty\right\}
\end{multline*}
and
\begin{equation*}\label{definition E infini}
\E^{\infty} = \bigcap_{\sigma\supeg 2} \E^{\sigma}
\end{equation*}
with the weight $p$ defined as
\begin{equation*}\label{definition p}
p(m,n):= \crochet{\crochet{m}^{-2}\frac{n}{m}},~(m,n)\in (\Z^*)^2
\end{equation*}
so that with this definition $\E = \E^2$.

Let $M\in 2^{\N}$. As in~\cite{IonescuKenigTataru2008}, for $b\in [0;1/2]$ the dyadic Bourgain type space is defined as
\begin{multline*}
X_M^{b} :=\left\{f(\tau,m,n)\in L^2(\R\times\Z^2), \supp f \subset \R \times I_M \times \Z, \right. \\ \left.\norme{X_M^{b}}{f}:= \sum_{K\supeg 1}K^{b}\normL{2}{\rho_K(\tau+\omega)f}<+\infty\right\}
\end{multline*}
When $b=1/2$ we simply write $X_M$.

Then, we use the $X_M^{b}$ structure uniformly on time intervals of size $M^{-2}$ :
\begin{multline*}
\fl_M^{b} := \left\{u(t,x,y)\in \mathcal{C}\left(\R,\E^{\infty}\right),~P_M u = u,\right.\\ \left.\norme{\fl_M^{b}}{u}:= \sup_{t_M \in \R} \norme{X_M^{b}}{p\cdot \F\left\{\chi_{M^{-2}}(t-t_M)u\right\}}<+\infty\right\}
\end{multline*}
and
\begin{multline*}
\nl_M := \left\{u(t,x,y)\in L^2\left(\R,\E^{\infty}\right),~P_M u = u,\right.\\ \left. \norme{\nl_M}{u}:= \sup_{t_M \in \R} \norme{X_M}{p\cdot \left|\tau+\omega +iM^2\right|^{-1} \F\left\{\chi_{M^{-2}}(t-t_M)u\right\}}<+\infty\right\}
\end{multline*}
For a function space $Y\hookrightarrow \mathcal{C}(\R,\E^{\infty})$, we set
\begin{equation*}
Y(T) := \left\{u\in\mathcal{C}\left([-T,T],\E^{\infty}\right),~\norme{Y(T)}{u}<+\infty\right\} 
\end{equation*}
endowed with
\begin{equation}\label{definition norme localisation T}
\norme{Y(T)}{u}:=\inf\left\{\norme{Y}{\widetilde{u}},~~\widetilde{u}\in Y,~~\widetilde{u}\equiv u \text{ on }[-T,T]\right\}
\end{equation}
Finally, the main function spaces are defined as
\begin{multline}\label{definition F}
\Fl^{\sigma,b}(T) := \left\{u\in\mathcal{C}([-T,T],\E^{\sigma}),\right. \\ \left.\norme{\Fl^{\sigma,b}(T)}{u}:= \left(\sum_{M\supeg 1} M^4\norme{\fl_M^{b}(T)}{P_M u}^2\right)^{1/2}<+\infty\right\}
\end{multline}
and
\begin{multline}\label{definition N}
\Nl^{\sigma}(T) := \left\{u\in L^2([-T,T],\E^{\sigma}),\right. \\ \left.\norme{\Nl^{\sigma}(T)}{u}:= \left(\sum_{M\supeg 1} M^4\norme{\nl_M(T)}{P_M u}^2\right)^{1/2}<+\infty\right\}
\end{multline}
The last space is the energy-type space which is the analogous in this context of the usual space $L^{\infty}([-T;T],\E^{\sigma})$ :
\begin{multline}\label{definition espace energie}
\Bl^{\sigma}(T) := \left\{u\in \mathcal{C}([-T,T],\E^{\sigma}),\right. \\ \left.\norme{\Bl^{\sigma}(T)}{u}:=\left(\sum_{M\supeg 1}\sup_{t_M\in [-T,T]}\norme{\E^{\sigma}}{P_M u(t_M)}^2\right)^{1/2}<+\infty\right\} 
\end{multline}
Again, for $\fl_M^{b}$ and $\Fl^{\sigma,b}(T)$, if $b=1/2$ we just drop it. We do the same for $\sigma=2$.

For the difference equation, we use similar spaces $\overline{F_M}$, $\overline{N_M}$ and $\overline{\Fl}(T)$, $\overline{\Nl}(T)$ and $\overline{\B}(T)$ which are the same as the above spaces but without the weight $p$ and at regularity $\sigma=0$. Let us notice that in view of the definition of $p$ we then have
\[\norme{F_M(T)}{u}^2\sim \norme{\overline{F_M}(T)}{u}^2+M^{-4}\norme{\overline{F_M}(T)}{\drx^{-1}\dry u}^2\]
\subsection{Basic properties}
We collect here some basic properties of the spaces $X_M$, $\Fl(T)$ and $\Nl(T)$. The proof of these results can be found e.g in \cite{IonescuKenigTataru2008,GuoOh2015,KenigPilod,Article1}.

First, for any $f_M\in X_M$, we have
\begin{equation}\label{estimation controle norme L2L1}
\norme{\ell^2_{m,n}L^1_{\tau}}{ f_M}\lesssim \norme{X_M}{f_M}
\end{equation}
Moreover, if we take $\gamma\in L^2(\R)$ satisfying 
\begin{equation}\label{hypothese localisation en temps X}
|\widehat{\gamma(\tau)}|\lesssim \crochet{\tau}^{-4}
\end{equation}
then for any $K_0\supeg 1$ and $t_0\in\R$ we have
\begin{multline}\label{propriete XM1}
K_0^{1/2}\normL{2}{\chi_{K_0}(\tau+\omega)\F\left\{\gamma(K_0(t-t_0))\F^{-1}f_M\right\}}\\+\sum_{K\supeg K_0}K^{1/2}\normL{2}{\rho_K(\tau+\omega)\F\left\{\gamma(K_0(t-t_0))\F^{-1}f_M\right\}}\lesssim \norme{X_M}{f_M}
\end{multline}
and the implicit constants are independent of $M$, $K_0$ and $t_0$.

For general time multipliers $m_M \in\mathcal{C}^4(\R)$ bounded along with its derivatives, as in \cite{IonescuKenigTataru2008} we have the bounds
\begin{equation}\label{estimation multiplicateur F}
\norme{F_M}{m_M(t)f_M}\lesssim \left(\sum_{k=0}^{4}(1\ou M)^{-k}\normL{\infty}{m_M^{(k)}}\right)\norme{F_M}{f_M}
\end{equation}
and
\begin{equation}\label{estimation multiplicateur N}
\norme{N_M}{m_M(t)f_M}\lesssim \left(\sum_{k=0}^{4}(1\ou M)^{-k}\normL{\infty}{m_M^{(k)}}\right)\norme{N_M^{b,b_1}}{f_M}
\end{equation}
We will also use \cite[Lemma 3.4]{GuoOh2015} to get a factor $T^{0+}$ in the estimates in order to avoid rescaling :
\begin{lemme}\label{lemme extraction T}
Let $T\in ]0;1]$ and $0\infeg b < 1/2$. Then, for any $u\in \fl_{M}(T)$,
\begin{equation}\label{estimation facteur T}
\norme{\fl_M^{b}(T)}{u}\lesssim T^{(1/2-b)-}\norme{\fl_M(T)}{u}
\end{equation}
and the implicit constant is independent of $M$ and $T$.
\end{lemme}
The last estimate justifies the use of $\Fl(T)$ as a resolution space :
\begin{lemme}
Let $\sigma\supeg 2$, $T\in ]0;1]$ and $u\in \Fl^{\sigma}(T)$. Then
\begin{equation}\label{estimation injection continue}
\norme{L^{\infty}_T\E^{\sigma}}{u}\lesssim \norme{\Fl^{\sigma}(T)}{u}
\end{equation}
\end{lemme}
\subsection{Linear estimate}\label{section estimation lineaire}
In this last subsection, we recall a linear estimate which replaces the usual estimate in the context of standard Bourgain spaces. The proof is the same as the one of \cite[Proposition 3.2]{IonescuKenigTataru2008}.
\newpage
\begin{proposition}
Let $T>0$ and $u,f\in \mathcal{C}([-T;T],\E^{\infty})$ satisfying 
\begin{equation}\label{equation lineaire inhomogene}
\drt u - \drx^5 u -\drx^{-1}\dry^2 u = f
\end{equation}
on $[-T,T]\times\T^2$.\\
Then for any $\sigma\supeg 2$, we have
\begin{equation}\label{equation estimation linéaire}
\norme{\Fl^{\sigma}(T)}{u}\lesssim \norme{\Bl^{\sigma}(T)}{u}+\norme{\Nl^{\sigma}(T)}{f}
\end{equation}
and
\begin{equation}\label{estimation lineaire difference}
\norme{\overline{\Fl}(T)}{u}\lesssim \norme{\overline{\Bl}(T)}{u}+\norme{\overline{\Nl}(T)}{f}
\end{equation}
\end{proposition}
\section{Dyadic estimates}\label{section estimation dyadique}
We prove here several estimates on the trilinear form $\int_{\R}\sum_{\Z^2}f_1\star f_2\cdot f_3$ which replace \cite[Corollary 5.3]{IonescuKenigTataru2008} in our context.

For the proof of the following easy lemmas, we refer to \cite[Section 3]{Article1}.
\begin{lemme}
Let $f_i\in L^2(\R\times\Z^2)$ be such that $\supp f_i \subset D_{M_i,\infeg K_i}\cap \R\times\Z\times I_i$, with $M_i,K_i\in 2^{\N}$ and $I_i\subset \Z$, $i=1,2,3$. Then
\begin{equation}\label{estimation forme trilineaire triviale}
\int_{\R}\sum_{\Z^2}f_1\star f_2 \cdot f_3 \lesssim M_{min}^{1/2}K_{min}^{1/2}(\#I_{min})^{1/2}\prod_{i=1}^3\normL{2}{f_i}
\end{equation}
\end{lemme}
\begin{lemme}\label{lemme mesure ensemble avec projections et sections}
Let $\Lambda \subset \Z^2$. We assume that the projection of $\Lambda$ on the $m$ axis is contained in an interval $I\subset \Z$. Moreover, we assume that the cardinal of the $n$-sections of $\Lambda$ (that is the sets $\left\{n \in \Z, (m_0,n)\in \Lambda\right\}$ for a fixed $m_0$) is uniformly (in $m_0$) bounded by a constant $C$. Then we have
\begin{equation*}
\left|\Lambda\right| \infeg C \crochet{|I|}
\end{equation*}
\end{lemme}
\newpage
\begin{lemme}\label{lemme mesure ensemble avec fonction}
Let $I$, $J$ be two intervals in $\R$, and let $\varphi : I \rightarrow \R$ be a $\mathcal{C}^1$ function with $\inf_{x\in J}\left|\varphi'(x)\right|>0$. Assume that $\left\{n\in J\cap \Z,~\varphi(n) \in I\right\}\neq\emptyset$. Then
\begin{equation}\label{estimation mesure cas Z}
\#\left\{n \in J\cap \Z,~\varphi(n) \in I\right\}\lesssim \crochet{\frac{|I|}{\inf_{x\in J}|\varphi'(x)|}}
\end{equation}
\end{lemme}
\begin{lemme}\label{lemme mesure parabole}
Let $a\neq 0$ ,$b$, $c$ be real numbers and $I\subset \R$ a bounded interval. Then
\begin{equation}\label{estimation mesure parabole Z}
\#\left\{n \in \Z,~an^2+bn+c \in I\right\} \lesssim \crochet{\frac{|I|^{1/2}}{|a|^{1/2}}}
\end{equation}
\end{lemme}
The main estimates of this section are the following :
\begin{proposition}\label{estimation Strichartz localisee}
Let $M_i,K_i\in 2^{\N}$, $i=1,2,3$, and take $u_1,u_2\in L^2(\R\times\Z^2)$ be such that $\supp (u_i)\subset D_{M_i,\infeg K_i}$. Then
\begin{multline}\label{estimation Strichartz grossiere}
\norme{L^2}{\mathbb{1}_{D_{M_3,\infeg K_3}}\cdot u_1\star u_2}\lesssim (K_1 \et K_2)^{1/2}M_{min}^{1/2}\\ \cdot \crochet{(K_1\ou K_2)^{1/4}(M_1\et M_2)^{1/4}}\normL{2}{u_1}\normL{2}{u_2}
\end{multline}
Moreover, if we are in the case $K_{max}\infeg 10^{-10}M_1M_2M_3M_{max}^2$, then
\begin{multline}\label{estimation basse modulation}
\norme{L^2}{\mathbb{1}_{D_{M_3,\infeg K_3}}\cdot u_1\star u_2}\lesssim (K_1\et K_2)^{1/2}M_{min}^{1/2}\\ \cdot \crochet{(K_1\ou K_2)^{1/2}(M_3M_{max})^{-1/2}}\normL{2}{u_1}\normL{2}{u_2}
\end{multline}
\end{proposition}
\begin{proof}
These estimates are the analogous of those proved in \cite[Subsections 2.1\& 2.2]{SautTzvetkov2001} in the context of the bilinear estimate in standard Bourgain spaces. The proof is very similar to that one of \cite[Proposition 5.5]{Article1}.
First, we split $u_1$ and $u_2$ depending on the value of $m_i$ on an $M_3$ scale, meaning
\begin{equation}\label{estimation decomposition echelle strichartz bilineaire}
\norme{L^2}{\mathbb{1}_{D_{M_3,\infeg K_3}}\cdot u_1\star u_2} \infeg \sum_{k\in\Z}\sum_{j\in\Z} \norme{L^2}{\mathbb{1}_{D_{M_3,\infeg K_3}}\cdot u_{1,k}\star u_{2,j}}
\end{equation} 
with
\begin{equation*}
u_{i,j} := \mathbb{1}_{[jM_3,(j+1)M_3]}(m_i)u_i
\end{equation*}
The conditions $|m|\sim M_3$, $m_1 \in [kM_3,(k+1)M_3]$ and $m-m_1 \in [j M_3;(j+1)M_3]$ require $j\in [-k-c;-k+c]$ for an absolute constant $c>0$. 

Squaring the norm in the right-hand side of (\ref{estimation decomposition echelle strichartz bilineaire}), it suffices to evaluate
\begin{multline*}
\int_{\R}\sum_{(m,n)\in\Z^2}\mathbb{1}_{D_{M_3,\infeg K_3}}(\tau,m,n)\left|\int_{\R}\sum_{(m_1,n_1)\in\Z^2} u_{1,k}(\tau_1,m_1,n_1)\right.\\
\left.\cdot u_{2,j}(\tau-\tau_1,m-m_1,n-n_1)\dtau_1\right|^2\dtau
\end{multline*}
 Using Cauchy-Schwarz inequality, the integral above is controled by
\begin{equation*}
\sup_{(\tau,m,n)\in D_{M_3,\infeg K_3}} \left|A_{\tau,m,n}\right|\cdot\normL{2}{u_{1,k}}^2\normL{2}{u_{2,j}}^2
\end{equation*}
where $A_{\tau,m,n}$ is defined as
\begin{multline*}
A_{\tau,m,n} := \left\{(\tau_1,m_1,n_1)\in\R\times\Z^2,m_1\in I_k,m-m_1\in I_j,\right.\\ \left.~\crochet{\tau_1 - \omega(m_1,n_1)}\lesssim K_1,~\crochet{\tau-\tau_1 - \omega(m-m_1,n-n_1)}\lesssim K_2\right\}
\end{multline*}
with the intervals \[I_k = I_{M_1}\cap [kM_3;(k+1)M_3]\text{ and }I_j = I_{M_2}\cap [jM_3;(j+1)M_3]\]
Using the triangle inequality in $\tau_1$, we get the bound
\[\left|A_{\tau,m,n}\right|\lesssim (K_1\et K_2) \left|B_{\tau,m,n}\right|\]
where $B_{\tau,m,n}$ is defined as
\begin{multline*}
B_{\tau,m,n} := \left\{(m_1,n_1)\in\Z^2,~m_1\in I_k,m-m_1\in I_j,\right.\\ 
\left.\crochet{\tau+\omega(m,n) -\Omega(m_1,n_1,m-m_1,n-n_1)}\lesssim (K_1\ou K_2)\right\}
\end{multline*}
and the resonant function $\Omega$ is defined as
\begin{multline}\label{definition fonction resonnance}
\Omega(m_1,n_1,m_2,n_2) = \omega(m_1,n_1)+\omega(m_2,n_2)-\omega(m_1+m_2,n_1+n_2)\\
=5m_1m_2(m_1+m_2)\alpha(m_1,m_2)-\frac{(m_1n_2-m_2n_1)^2}{m_1m_2(m_1+m_2)}\\
=5m_1m_2(m_1+m_2)\alpha(m_1,m_2)-\frac{m_1m_2}{m_1+m_2}\left(\frac{n_1}{m_1}-\frac{n_2}{m_2}\right)^2
\end{multline}  
with 
\begin{equation*}
\alpha(m_1,m_2)=m_1^2+m_1m_2+m_2^2 \sim M_{max}^2
\end{equation*}
First, in the case $K_{max}\infeg 10^{-10}M_1M_2M_3M_{max}^2$, we estimate $\left|B_{\tau,m,n}\right|$ with the help of Lemma~\ref{lemme mesure ensemble avec projections et sections} and~\ref{lemme mesure ensemble avec fonction}. Indeed, its projection on the $m_1$ axis is controled by $|I_k|\et |I_j|$. Now, we compute
\begin{multline*}
\left|\frac{\partial \Omega}{\partial n_1}\right| =2\left|\frac{n_1}{m_1}-\frac{n-n_1}{m-m_1}\right|\\
=2\left|\frac{m}{m_1(m-m_1)}\left(5m_1(m-m_1)m\alpha(m_1,m-m_1)-\Omega\right)\right|^{1/2}
\end{multline*}
Thus, from the condition $|\Omega|\lesssim K_{max}\infeg 10^{-10}M_1M_2M_3M_{max}^2$ we get 
\[\left|\frac{\partial \Omega}{\partial n_1}\right|\gtrsim \left|\frac{m}{m_1(m-m_1)}\cdot m_1(m-m_1)m\alpha(m_1,m-m_1)\right|^{1/2}\sim M_3M_{max}\]
So we can estimate $|B_{\tau,m,n}|$ in this regime by 
\[|B_{\tau,m,n}| \lesssim \crochet{|I_k|\et|I_j|}\crochet{(K_1\ou K_2)(M_3M_{max})^{-1}}\]
For (\ref{estimation Strichartz grossiere}), note that we can neglect the localization ${\displaystyle \mathbb{1}_{D_{M_3,\infeg K_3}}}$, thus we can use the argument of \cite[Lemma 4]{SautTzvetkov2001} and assume that $m_i\supeg 0$ on the support of $u_i$. To get a bound for $|B_{\tau,m,n}|$, we now use Lemma~\ref{lemme mesure parabole} instead of Lemma~\ref{lemme mesure ensemble avec fonction}. Indeed, we can write $\tau -\omega(m,n)-\Omega(m_1,n_1,m-m_1,n-n_1)$ as
\begin{multline*}
\tau - \omega(m,n)-5mm_1(m-m_1)\alpha(m_1,m-m_1)\\
+\frac{m_1^2n^2-2m_1mn_1n}{m_1m(m-m_1)}+\frac{m^2}{m_1m(m-m_1)}n_1^2
\end{multline*} 
which is a parabola in $n_1$ with leading coefficient
\[\left|\frac{m}{m_1(m-m_1)}\right|=\frac{1}{m_1}+\frac{1}{m-m_1}\supeg \frac{1}{m_1\et (m-m_1)}\]
Thus for a fixed $m_1$, the cardinal of the $n_1$-section is estimated by \[\crochet{(K_1\ou K_2)^{1/2}(M_1\et M_2)^{1/2}}\] thanks to (\ref{estimation mesure parabole Z}). So we get the final bound
\[|B_{\tau,m,n}| \lesssim \crochet{|I_k|\et|I_j|}\crochet{(K_1\ou K_2)^{1/2}M_{min}^{1/2}}\]
These bounds for $|A_{\tau,m,n}|$ finally give (\ref{estimation Strichartz grossiere}) and (\ref{estimation basse modulation}) by using Cauchy-Schwarz inequality to sum over $k\in\Z$, since $|I_k| \lesssim M_1\et M_3$ and $|I_j|\lesssim M_2\et M_3$.
\end{proof}
\begin{remarque}
In the context of standard Bourgain spaces, we cannot recover some derivatives in the regime $K_{max}<M_3M_{max}$ since \[\crochet{(K_1\ou K_2)^{1/2}(M_3M_{max})^{-1/2}}=1\] in that case. This is the main reason for the bilinear estimate to fail in \cite[Section 5]{SautTzvetkov2001} and for our choice of time localization on intervals of size $M_{max}^{-2}$.
\end{remarque}
\begin{remarque}\label{remarque gap}
Estimate (\ref{estimation basse modulation}) may seem rough, but a more carefull analysis of the dyadic bilinear estimates in the resonant case (that is, the analogous of\\ \cite[Lemma 3.1 (a)]{Guo2017} for periodic functions) in the spirit of \cite[Lemma 3.1]{Zhang2015} leads to the bound
\begin{multline*}
(K_{1}K_{3})^{1/2}M_{max}^{-1}\\
\cdot\left\{\left(\frac{K_{2}}{(M_1\et M_2)M_{max}}\right)^{1/2}\et \left[(M_1\et M_2)\crochet{\frac{K_{2}}{(M_{min}M_{max}^3}}\right]^{1/2}\right\}
\end{multline*}
showing that, in the case $K_2=K_{med}\infeg M_{min}M_{max}^3$ and $M_1\et M_2 = M_{min}$, (\ref{estimation basse modulation}) is actually optimal. Comparing with \cite[Lemma 3.1]{Guo2017}, we see why there is such a gap in regularity between the well-posedness in $\R^2$ and $\T^2$.
\end{remarque}
As in \cite[Corollary 5.3]{IonescuKenigTataru2008}, we conclude this section by summerizing the main dyadic estimates that we will use  throughout the forthcoming sections.
\begin{corollaire}\label{corollaire synthese}
Assume $M_1,M_2,M_3,K_1,K_2,K_3\in 2^{\N}$ with $K_i\supeg M_i^2$ and $f_i\in L^2(\R\times\Z^2)$ are positive functions with the support condition $\supp f_i \subset D_{M_i,K_i}$, $i=1,2$. Then
\begin{equation}\label{estimation trilineaire synthese}
\normL{2}{\mathbb{1}_{D_{M_3,K_3}} \cdot f_1\star f_2}\lesssim M_{min}^{1/2}M_{max}^{-1-2b}(K_{min}K_{max})^{1/2}K_{med}^b \normL{2}{f_1}\normL{2}{f_2}
\end{equation}
for any $b\in[1/4;1/2]$, and
\begin{equation}\label{estimation dyadique poids}
\normL{2}{\mathbb{1}_{D_{M_3,K_3}} \cdot f_1\star f_2}\lesssim M_1^{3/2}M_{min}^{1/2} K_{min}^{1/2} \normL{2}{p\cdot f_1}\normL{2}{f_2}
\end{equation}
\end{corollaire}
\begin{proof}
(\ref{estimation trilineaire synthese}) follows directly from (\ref{estimation Strichartz grossiere}) and (\ref{estimation basse modulation}) above. 

For the proof of (\ref{estimation dyadique poids}), we follow \cite[Lemma 5.3]{IonescuKenigTataru2008} : we split
\begin{equation*}
f_1= \sum_{N\supeg M_1^3}f_{1,N} = \mathbb{1}_{I_{\infeg M_1^3}}(n)f_1 + \sum_{N>M_1^3}\mathbb{1}_{I_N}(n)f_1
\end{equation*}
such that
\begin{equation*}
\normL{2}{\mathbb{1}_{D_{M_3,\infeg K_3}} \cdot f_1\star f_2} \lesssim \sum_{N\supeg M_1^3}N^{1/2}M_{min}^{1/2}K_{min}^{1/2}\normL{2}{f_{1,N}}\normL{2}{f_2}
\end{equation*}
after using (\ref{estimation forme trilineaire triviale}).\\
Thus, using Cauchy-Schwarz inequality in $N$, we obtain
\begin{multline*}
\normL{2}{\mathbb{1}_{D_{M_3,K_3}} \cdot f_1\star f_2} \lesssim M_{min}^{1/2}K_{min}^{1/2}\normL{2}{f_2}\sum_{N\supeg M_1^3}N^{-1/2}M_1^3\normL{2}{p \cdot f_{1,N}}\\
\lesssim M_1^{3/2}M_{min}^{1/2}K_{min}^{1/2}\normL{2}{p\cdot f_1}\normL{2}{f_2}
\end{multline*}
\end{proof}

\section{Energy estimates}\label{section estimation energie}
In this section, we prove the energy estimates which allow to control the $\B$-norm of regular solutions and the $\overline{\B}$-norm of the difference of solutions.
\begin{lemme}\label{lemme estimation trilineaire}
There exists $\mu_0>0$ small enough such that for $T\in]0;1]$ and $u_i\in \overline{\fl_{M_i}}(T)$, $i\in\{1,2,3\}$, with one of them in $\fl_{M_i}(T)$, then
\begin{equation}\label{estimation energie forme trilineaire}
\left|\int_{[0,T]\times\T^2}u_1u_2u_3 \dt\dx\dy \right|\lesssim T^{\mu_0}M_{min}^{1/2}\prod_{i=1}^3 \norme{\overline{\fl_{M_i}}(T)}{u_i}
\end{equation}
If moreover $M_1\infeg M/16$, and $u\in \overline{\fl_M}(T)$, $v\in \fl_{M_1}(T)$, we have
\begin{multline}\label{estimation terme specifique estimation energie}
\left|\int_{[0,T]\times\T^2}P_M u\cdot P_M (P_{M_1}v\cdot\drx u)\dt\dx\dy \right|\\
 \lesssim T^{\mu_0} M_1^{3/2}\norme{\overline{\fl_{M_1}}(T)}{P_{M_1}v}\sum_{M_2\sim M}\norme{\overline{\fl_{M_2}}(T)}{P_{M_2} u}^2
\end{multline}
\end{lemme}
\begin{proof}
From symmetry, we may assume $M_1\infeg M_2\infeg M_3$. Let $\widetilde{u_i}\in \overline{\fl_{M_i}}$ be extensions $u_i$ to $\R$, satisfying ${\displaystyle \norme{\fl_{M_i}}{\widetilde{u_i}}\infeg 2\norme{\fl_{M_i}(T)}{u_i}}$.\\
Let $\gamma\in \mathcal{C}^{\infty}_c(\R)$ be such that $\gamma : \R\rightarrow [0;1]$ with $\supp \gamma \subset [-1;1]$ and satisfying
\[\forall t\in\R,~\sum_{\nu\in\Z}\gamma^3(t-\nu) = 1\]
Then
\begin{multline*}
\int_{[0;T]\times\T^2}u_1u_2u_3\dt\dx\dy\\
 \lesssim  \sum_{|\nu|\lesssim M_{max}^2}\sum_{K_1,K_2,K_3\supeg M_{max}^2}\int_{\R\times\Z^2}\left(\rho_{K_3}(\tau+\omega)\F\left\{\mathbb{1}_{[0;T]}\gamma(M_{max}^2t-\nu)\widetilde{u_3}\right\}\right)\\
\cdot \left(\rho_{K_1}(\tau+\omega)\F\left\{\mathbb{1}_{[0;T]}\gamma(M_{max}^2t-\nu)\widetilde{u_1}\right\}\right)\\
\star \left(\rho_{K_2}(\tau+\omega)\F\left\{\mathbb{1}_{[0;T]}\gamma(M_{max}^2t-\nu)\widetilde{u_2}\right\}\right)\dtau 
\end{multline*}
where there are at most $TM_{max}^2$ interior terms $\nu$ for which $\mathbb{1}_{[0;T]}\gamma(M_{max}^2t-\nu)=\gamma(M_{max}^2t-\nu)$, and at most 4 remaining border terms where the integral is non zero. The property of $X_M$ (\ref{propriete XM1}) allows us to partition the modulations at $K_i\supeg M_{max}^2$

Let us now observe that, using (\ref{propriete XM1}), for the interior terms we have
\[\sup_{\nu\in\Z}\sum_{K_i\supeg M_{max}^2}K_i^{b}\normL{2}{\rho_{K_i}(\tau+\omega)\F\left\{\gamma(M_{max}^2t-\nu)\widetilde{u_i}\right\}}\lesssim \norme{\overline{\fl_{M_i}}^{b}}{\widetilde{u_i}}\]
Thus, since we can take $\mu_0=1$ for those terms, (\ref{estimation energie forme trilineaire}) follows from (\ref{estimation trilineaire synthese}) with $b=1/2$ and the estimate above.

For the remaining border terms, we use that
\[\sup_{\nu}\sup_{K_i\supeg M_{max}^2}K_i^{1/2}\normL{2}{\rho_{K_i}(\tau+\omega)\cdot\widehat{\mathbb{1}_{[0;T]}}\star\F\left\{\gamma(M_{max}^2t-\nu)\widetilde{u_i}\right\}}\lesssim \norme{\overline{\fl_{M_i}}}{\widetilde{u_i}}\]
which follows through the same argument as for the proof of (\ref{propriete XM1}) (see \cite{Article1}). Thus we can use (\ref{estimation trilineaire synthese}) with $b<1/2$ to get (\ref{estimation energie forme trilineaire}).

(\ref{estimation terme specifique estimation energie}) then follows from the one of (\ref{estimation energie forme trilineaire}) through the same argument as in \cite[Lemma 6.1]{IonescuKenigTataru2008}.
\end{proof}
We can now state our global energy estimate.
\begin{proposition}\label{proposition estimation energie}
Let $T\in ]0;1[$ and $u\in\mathcal{C}([-T,T],\El^{\infty})$ be a solution of (\ref{equation KP1-5}) on $[-T,T]$. Then for any $\sigma \supeg 2$,
\begin{equation}\label{estimation energie globale}
\norme{\Bl^{\sigma}(T)}{u}^2\lesssim \norme{\El^{\sigma}}{u_0}^2 + T^{\mu_0}\norme{\Fl(T)}{u}\norme{\Fl^{\sigma}(T)}{u}^2
\end{equation}
\end{proposition}
\begin{proof}
From the definition of the $\B^{\sigma}$ norm and the weight $p$, we have the first estimate
\begin{multline*}
\norme{\Bl^{\sigma}(T)}{u}^2 \\
\lesssim \sum_{M_3> 1}\sup_{t_{M_3}\in[-T;T]}\left(M_3^{2\sigma}\normL{2}{P_{M_3}u(t_{M_3})}+M_3^{2(\sigma-2)}\normL{2}{\drx^{-1}\dry P_{M_3}u(t_{M_3})}\right)
\end{multline*}
For the first term within the sum, using that $u$ is a solution to (\ref{equation KP1-5}), we have
\begin{multline*}
\sup_{t_{M_3}\in[-T;T]}M_3^{2\sigma}\normL{2}{P_{M_3}u(t_{M_3})}\lesssim M_3^{2\sigma}\normL{2}{P_{M_3}u_0}\\+M_3^{2\sigma}\left|\int_{[0;T]\times\T^2}P_{M_3}u\cdot P_{M_3}(u\drx u)\dt\dx\dy\right|
\end{multline*}
We can divide the previous integral term into
\begin{multline*}
\sum_{M_1\infeg M_3/16}M_3^{2\sigma}\left|\int_{[0;T]\times\T^2}P_{M_3}u\cdot P_{M_3}(P_{M_1}u\cdot\drx u)\dt\dx\dy\right|\\
+\sum_{M_1\gtrsim M_3}\sum_{M_2\supeg 1}M_3^{2\sigma}\left|\int_{[0;T]\times\T^2}P_{M_3}^2u\cdot P_{M_1}u\cdot\drx P_{M_2} u\dt\dx\dy\right|
\end{multline*}
Using (\ref{estimation terme specifique estimation energie}) for the first one and (\ref{estimation energie forme trilineaire}) for the second one, we get the bound
\begin{multline*}
\sum_{M_1\infeg M_3/16}M_3^{2\sigma}M_1^{3/2}\norme{\overline{\fl_{M_1}}(T)}{P_{M_1}u}\sum_{M_2\sim M_3}\norme{\overline{\fl_{M_2}}(T)}{P_{M_2}u}^2\\
+\sum_{M_1\gtrsim M_3}\sum_{M_2\supeg 1}M_3^{2\sigma}M_2(M_2\et M_3)^{1/2}\\
\cdot\norme{\overline{\fl_{M_1}}(T)}{P_{M_1}u}\norme{\overline{\fl_{M_2}}(T)}{P_{M_2}u}\norme{\overline{\fl_{M_3}}(T)}{P_{M_3}u}
\end{multline*}
For the sum on the first line, we use Cauchy-Schwarz inequality to sum on $M_1$ (as we have 1/2 derivative to spare) and then sum on $M_3$ by writing $M_2 = 2^k M_3$ with $k\in\Z$ bounded and then a use of Cauchy-Schwarz inequality in $M_3$.\\
For the second line, we cut the sum into two parts $M_2 \gtrsim M_3\sim M_1$ and $M_2\sim M_1\gtrsim M_3$, put $2\sigma$ derivatives on the highest frequency, and then use Cauchy-Schwarz again to sum on the lowest frequency (we have again 1/2 extra derivative) and then the biggest. Thus the term above is bounded by the right-hand side of (\ref{estimation energie globale}).

It remains to treat the sum with the antiderivative. Proceeding similarly and writing $v:=\drx^{-1}\dry u$, we get
\begin{multline*}
\sup_{t_{M_3}\in[-T;T]}M_3^{2(\sigma-2)}\normL{2}{\drx^{-1}\dry P_{M_3}u(t_{M_3})}\lesssim M_3^{2(\sigma-2)}\normL{2}{\drx^{-1}\dry P_{M_3}u_0}\\+M_3^{2(\sigma-2)}\left|\int_{[0;T]\times\T^2}P_{M_3}v\cdot P_{M_3}(u\drx v)\dt\dx\dy\right|
\end{multline*}
which is analogously dominated by
\begin{multline*}
\sum_{M_1\infeg M_3/16}M_3^{2(\sigma-2)}M_1^{3/2}\norme{\overline{\fl_{M_1}}(T)}{P_{M_1}u}\sum_{M_2\sim M_3}\norme{\overline{\fl_{M_2}}(T)}{P_{M_2}v}^2\\
+\sum_{M_1\gtrsim M_3}\sum_{M_2\supeg 1}M_3^{2(\sigma-2)}M_2(M_2\et M_3)^{1/2}\\
\cdot\norme{\overline{\fl_{M_1}}(T)}{P_{M_1}u}\norme{\overline{\fl_{M_2}}(T)}{P_{M_2}v}\norme{\overline{\fl_{M_3}}(T)}{P_{M_3}v}
\end{multline*}
For the first line, we run the summation over $M_1,M_2,M_3$ as before, whereas for the second line, we split the highest frequency into $M_1^2(M_2\ou M_3)^{2(\sigma-3)}(M_2\et M_3)^{3/2}$ and then perform the summation as above.
\end{proof}
\begin{remarque}
In the dyadic summations above, we see that we are 1/2-derivative below the energy space, thus a simple adaptation of our argument would actually yield local well-posedness in $H^{s_1,s_2}(\T^2)$ with $s_1>3/2, s_2\supeg 0$. For our result to be more readable, we chose not to present these technical details here.
\end{remarque}
\begin{remarque}
Even with the local well-posedness result mentioned above, our result is in sharp contrast with the local well-posedness of \cite{Guo2017} in the case of $\R^2$. This highilights the quasilinear behaviour of equation (\ref{equation KP1-5}) in the periodic setting. From the technical point of view, the $X_M$ structure is used in \cite{Guo2017} on time intervals on size $M^{-1}$, whereas in our case, the use of the counting measure instead of the Lebesgue measure in the localized bilinear Strichartz estimates requires us to work on time intervals of size $M^{-2}$ which explains the gap in regularity between these results.
\end{remarque}
To deal with the difference of solutions, we also prove the following proposition.
\begin{proposition}\label{proposition estimation energie difference}
Assume $T\in ]0;1[$ and $u,v\in\mathcal{C}([-T,T],\El^{\infty})$ are solutions to (\ref{equation KP1-5}) on $[-T,T]$ with initial data $u_0,v_0\in\E^{\infty}$. Then
\begin{equation}\label{estimation energie difference L2}
\norme{\overline{\Bl}(T)}{u-v}^2\lesssim \norme{L^2}{u_0-v_0}^2 + T^{\mu_0}\norme{\Fl(T)}{u+v}\norme{\overline{\Fl}(T)}{u-v}^2
\end{equation}
and
\begin{equation}\label{estimation energie difference globale}
\norme{\Bl(T)}{u-v}^2\lesssim \norme{\El}{u_0-v_0}^2 + T^{\mu_0}\norme{\Fl^3(T)}{v}\norme{\Fl(T)}{u-v}^2
\end{equation}
\end{proposition}
\begin{proof}
We proceed as in the previous proposition, except that now $w:=u-v$ solves the equation
\begin{equation}\label{equation difference}
\begin{cases}
{\displaystyle \drt w - \drx^5w -\drx^{-1}\dry^2w +\drx\left(w\frac{u+v}{2}\right)=0}\\
w(t=0)=u_0-v_0
\end{cases}
\end{equation}
For (\ref{estimation energie difference L2}), we write
\begin{multline*}
\norme{\overline{\Bl}(T)}{u-v}^2=\sum_{M_3\supeg 1}\sup_{t_{M_3}\in\R}\normL{2}{P_{M_3}(u-v)(t_{M_3})}^2
\lesssim \sum_{M_3\supeg 1}\left\{\normL{2}{P_{M_3}(u_0-v_0)}^2\right.\\
\left.+\left|\int_{[0;T]\times\T^2}P_{M_3}w\cdot P_{M_3}\left(w\drx w + w\drx v +v\drx w\right)\dt\dx\dy\right|\right\}
\end{multline*}
The first integral term with $w\drx w$ can be estimated by $\norme{\Fl(T)}{w}\norme{\overline{\Fl}}{w}^2$ the exact same way as the first term in the previous proposition with $\sigma=0$.

As in \cite{Zhang2015}, for the other two terms, we use again (\ref{estimation energie forme trilineaire}) and (\ref{estimation terme specifique estimation energie}) to bound them with
\begin{multline*}
T^{\mu_0}\left\{\sum_{M_3\supeg 1}\sum_{M_1\infeg M_3/16}\sum_{M_2\sim M_3}\left(M_2M_1^{1/2}\Pi_1 + M_1^{3/2}\Pi_2\right)\right.\\
+ \sum_{M_3\supeg 1}\sum_{M_1\gtrsim M_3}\sum_{M_2\sim M_1}M_2M_3^{1/2}\left(\Pi_1+\Pi_2\right)\\
\left. + \sum_{M_3\supeg 1}\sum_{M_1\sim M_3}\sum_{M_2\lesssim M_3}M_2^{3/2}\left(\Pi_1+\Pi_2\right)\right\}
\end{multline*}
where we have noted
\[\Pi_1 := \norme{F_{M_3}}{w}\norme{F_{M_1}}{w}\norme{F_{M_2}}{v}\text{ and }\Pi_2 := \norme{F_{M_3}}{w}\norme{F_{M_1}}{v}\norme{F_{M_2}}{w}\]
Observe that, as for Proposition~\ref{proposition estimation energie} above, we have 1/2 derivative to spare. Moreover, using the relation between the $M_i$'s, we can always place all $(3/2)+$ derivatives on the term containing $v$, thus we can sum by using Cauchy-Schwarz to bound all these terms with the right-hand side of (\ref{estimation energie difference L2}).

By the same token as for (\ref{estimation energie globale}), we can estimate the left-hand side of (\ref{estimation energie difference globale}) by ${\displaystyle \norme{\E}{u_0-v_0}^2}$ plus two integral terms
\begin{equation}\label{definition terme 1 estimation energie difference}
\sum_{M_3\supeg 1}M_3^4\left|\int_{[0;T]\times\T^2}P_{M_3}w\cdot P_{M_3}\left(w\drx w + w\drx v +v\drx w\right)\dt\dx\dy\right|
\end{equation}
and
\begin{equation}\label{definition terme 2 estimation energie difference}
\sum_{M_3\supeg 1}\left|\int_{[0;T]\times\T^2}P_{M_3}W\cdot P_{M_3}\left(w\drx W + w\drx V +v\drx W\right)\dt\dx\dy\right|
\end{equation}
where in the latter ${\displaystyle W:=\drx^{-1}\dry w}$ and ${\displaystyle V:=\drx^{-1}\dry v}$.

For (\ref{definition terme 1 estimation energie difference}), we proceed exactly as previously. Again, the first integral term has already been treated in the proof of (\ref{estimation energie globale}). For the other terms, now we have $11/2$ derivatives to distribute, and using again the relation between the $M_i$'s we can place $2^-$ derivatives on each $w$ and the remaining ones on $v$ and then run the summations, the worst case being the term $P_{M_1}w\cdot\drx P_{M_2}v$ in the regime $M_1\ll M_2\sim M_3$ since there are 5 highest derivatives, thus we need to put 3 on $v$.

It remains to treat (\ref{definition terme 2 estimation energie difference}). Once again, the first term within the integral appeared in the proof of the previous proposition, thus we only need to deal with the last two terms. We proceed the same way as above, since there are $3/2$ derivatives to share, and $w$ can absorb 2, $V$ can absorb 1 and $v$ can absorb 3. 
\end{proof}
\section{Short-time bilinear estimates}\label{section estimation bilineaire}
The aim of this section is to prove the bilinear estimates for both the equation and the difference equation. We mainly adapt \cite{IonescuKenigTataru2008}.
\begin{proposition}\label{propositioin estimation bilineaire equation}
There exists $\mu_1>0$ small enough such that for any $T\in]0;1]$ and $\sigma\supeg 2$ and $u,v\in \Fl^{\sigma}(T)$, 
\begin{equation}\label{estimation bilineaire globale}
\norme{\Nl^{\sigma}(T)}{\drx(uv)}\lesssim T^{\mu_1}\left\{\norme{\Fl^{\sigma}(T)}{u}\norme{\Fl(T)}{v}+\norme{\Fl(T)}{u}\norme{\Fl^{\sigma}(T)}{v}\right\}
\end{equation}
\end{proposition}
\begin{proof}
Using the definition of $\Fl^{\sigma}(T)$ (\ref{definition F}) and $\Nl^{\sigma}(T)$ (\ref{definition N}), the left-hand side of (\ref{estimation bilineaire globale}) is bounded by
\[\sum_{M_1,M_2,M_3}M_3^{\sigma}\norme{\nl_{M_3}(T)}{P_{M_3}\drx\left(P_{M_1}u\cdot P_{M_2}v\right)}\]
For $M_1,M_2\in 2^{\N}$, let us choose extensions $u_{M_1}$ and $v_{M_2}$ of $P_{M_1}u$ and $P_{M_2}v$ to $\R$ satisfying ${\displaystyle \norme{\fl_{M_1}}{u_{M_1}}\infeg 2\norme{\fl_{M_1}(T)}{P_{M_1}u}}$ and similarly for $v_{M_2}$. Since the previous term is symmetrical with respect to $u$ and $v$, we can assume $M_1\infeg M_2$. 

To treat the term above, from the definition of the $\fl_{M}^b$ and $\nl_{M}$ norms, the property of the space $X_M$ (\ref{propriete XM1}) and the use of Lemma~\ref{lemme extraction T}, it suffices to show that there exists $b\in [0;1/2)$ such that for all $K_i\supeg M_i^2$, $i=1,2,3$ and $f_i^{K_i}\in L^2(\R\times\Z^2)$ with $\supp f_i^{K_i} \subset D_{M_i,\infeg K_i}$, $i=1,2$ then
\begin{multline}\label{estimation bilineaire cle}
\frac{M_{max}^2}{M_3^2}\cdot M_3^{\sigma+1}\sum_{K_3\supeg M_3^2}K_3^{-1/2}\normL{2}{\mathbb{1}_{D_{M_3,\infeg K_3}} \cdot p \cdot f_1^{K_1}\star f_2^{K_2}}\\ \lesssim M_1^2M_2^{\sigma}(K_1\et K_2)^b(K_1\ou K_2)^{1/2}\normL{2}{p\cdot f_1^{K_1}}\normL{2}{p\cdot f_2^{K_2}}
\end{multline}
Indeed, for a smooth partition of unity $\gamma : \R\rightarrow [0;1]$ satisfying $\supp \gamma \subset [-1;1]$ and for all $t\in\R$
\[\sum_{\nu\in\Z}\gamma(t-\nu)^2=1\]
then define for $|\nu|\lesssim M_{max}^2M_3^{-2}$
\[f_{1,\nu}^{K_1} := \rho_{K_1}(\tau+\omega)\cdot\F\left\{\gamma\left(M_{max}^2M_3^{-2}t-\nu\right)u_{M_1}\right\}\]
with $\rho_{K_1}$ a non-homogeneous dyadic decomposition of unity partitioned at $K_1 = M_1^2$, and similarly for $f_{2,\nu}^{K_2}$. Then the norm within the sum is bounded by the left-hand side of (\ref{estimation bilineaire cle}) (after taking the supremum over $\nu$), wherease summing on $K_i\supeg M_i^2$, $i=1,2$, using (\ref{propriete XM1}) and Lemma~\ref{lemme extraction T} and summing on $M_1,M_2$ then the right-hand side of (\ref{estimation bilineaire cle}) is controled by the right-hand side of (\ref{estimation bilineaire globale}) (see e.g \cite{Article1} for the full details).

We then separate two cases depending on the relation between the $M_i$'s.\\
\textbf{Case A : Low $\times$ High $\rightarrow$ High.}\\ 
We assume $M_1\lesssim M_2\sim M_3$. In that case, for (\ref{estimation bilineaire cle}) it is sufficient to prove
\begin{multline}\label{estimation bilineaire lhh cle}
M_3\sum_{K_3\supeg M_3^2}K_3^{-1/2}\normL{2}{\mathbb{1}_{D_{M_3,\infeg K_3}} \cdot p \cdot f_1^{K_1}\star f_2^{K_2}}\\ \lesssim \ln(M_{min}) M_{min}^{1/2}M_{max}^{-2b}(K_1\et K_2)^b(K_1\ou K_2)^{1/2}\normL{2}{p\cdot f_1^{K_1}}\normL{2}{p\cdot f_2^{K_2}}
\end{multline}
for a $b\in[1/4;1/2)$.

Since $K_i\supeg M_i^2$, $i=1,2$, then for the large modulations, we combine (\ref{estimation dyadique poids})(for both $f_1$ and $f_2$) with the obvious bound
\begin{equation}\label{estimation bilineaire lhh poids}
p(m_1+m_2,n_1+n_2) \lesssim M_1^3M_3^{-3}p(m_1,n_1) + p(m_2,n_2)
\end{equation}
to bound the sum for $K_3\supeg M_3^2M_1^3$ by
\[M_1^{1/2}M_3^{-2b}(K_1\et K_2)^{1/2}(K_1\ou K_2)^b\normL{2}{p\cdot f_1^{K_1}}\normL{2}{p\cdot f_2^{K_2}}\]
for any $b\in [0;1/2]$.

For the small modulations $M_3^2\infeg K_3\infeg M_3^2M_1^3$, the sum runs over about $\ln(M_1)$ dyadic integers. Moreover, using the definition of $\Omega$ (\ref{definition fonction resonnance}), we can replace (\ref{estimation bilineaire lhh poids}) with
\begin{equation}\label{estimation bilineaire lhh poids 2}
p(m_1+m_2,n_1+n_2)\lesssim p(m_2,n_2)+M_1^{1/2}M_3^{-3}K_{max}^{1/2}
\end{equation}
Indeed, this follows from the definition of $\Omega$ which implies
\[\frac{|n|}{|m|}\lesssim \frac{|n_2|}{|m_2|}+\left(\frac{|m_1|}{|m_2m|}|\Omega| + 5m_1^2\alpha(m,m_2)\right)^{1/2}\]
In the case $K_{max}\infeg 10^{-10}M_1M_2M_3M_{max}^2$, we then use (\ref{estimation bilineaire lhh poids 2}), use the bound on $K_{max}$ and then use (\ref{estimation basse modulation}) to get the estimate
\begin{multline*}
\normL{2}{\mathbb{1}_{D_{M_3,\infeg K_3}} \cdot p \cdot f_1^{K_1}\star f_2^{K_2}}\lesssim \normL{2}{\mathbb{1}_{D_{M_3,\infeg K_3}} \cdot (p \cdot f_1^{K_1})\star (p\cdot f_2^{K_2})}\\
\lesssim (K_{min}K_{max})^{1/2}K_{med}^b M_{min}^{1/2}M_{max}^{-1-2b}\normL{2}{p\cdot f_1^{K_1}}\normL{2}{p\cdot f_2^{K_2}}
\end{multline*}
for any $b\in [0;1/2]$. Thus the sum in this regime is estimated with
\[\ln(M_1)M_{min}^{1/2}M_{max}^{-2b}(K_1\et K_2)^b(K_1\ou K_2)^{1/2}\normL{2}{p\cdot f_1^{K_1}}\normL{2}{p\cdot f_2^{K_2}}\] which suffices for (\ref{estimation bilineaire lhh cle}).\\
In the regime $K_{max}\gtrsim M_1M_2M_3M_{max}^2$, we apply again (\ref{estimation bilineaire lhh poids 2}), loose a factor $K_{max}^{1/2}$ in the first term, and then use (\ref{estimation Strichartz grossiere}) instead of (\ref{estimation basse modulation}) to obtain
\begin{multline*}
\normL{2}{\mathbb{1}_{D_{M_3,\infeg K_3}} \cdot p \cdot f_1^{K_1}\star f_2^{K_2}}\\
\lesssim K_{max}^{1/2}M_{min}^{-1/2}M_{max}^{-2}\normL{2}{\mathbb{1}_{D_{M_3,\infeg K_3}} \cdot (p \cdot f_1^{K_1})\star (p\cdot f_2^{K_2})}\\
\lesssim (K_{min}K_{max})^{1/2}K_{med}^b M_{max}^{-5/4-2b}\normL{2}{p\cdot f_1^{K_1}}\normL{2}{p\cdot f_2^{K_2}}
\end{multline*}
for any $b\in[1/4;1/2]$ which is controled by the estimate in the previous regime.\\
\textbf{Case B : High $\times$ High $\rightarrow$ Low.}\\
We assume now $M_1\sim M_2 \gtrsim M_3$. (\ref{estimation bilineaire lhh cle}) becomes in this case
\begin{multline}\label{estimation bilineaire hhl cle}
M_{max}^2M_3\sum_{K_3\supeg M_3^2}K_3^{-1/2}\normL{2}{\mathbb{1}_{D_{M_3,\infeg K_3}}\cdot p \cdot f_1^{K_1}\star f_2^{K_2}}\\
\lesssim \ln(M_{max})M_{max}^{3-2b}M_{min}^{-1/2}(K_1\et K_2)^b(K_1\ou K_2)^{1/2}\normL{2}{p\cdot f_1^{K_1}}\normL{2}{p\cdot f_2^{K_2}}
\end{multline}
For the high modulations $K_3\supeg M_{min}^{-2} M_{max}^{6}$, we use (\ref{estimation Strichartz grossiere}) along with the obvious bound
\begin{equation}\label{estimation bilineaire hhl poids}
p(m_1+m_2,n_1+n_2) \lesssim \frac{M_{max}^3}{M_{min}^3}\left(p(m_1,n_1)+p(m_2,n_2)\right)
\end{equation}
to estimate the left-hand side of (\ref{estimation bilineaire hhl cle}) with
\[M_{max}^2M_{min}^{-1/2}(K_1\et K_2)^{1/2}(K_1\ou K_2)^b \normL{2}{p\cdot f_1^{K_1}}\normL{2}{p\cdot f_2^{K_2}}
\]
for any $b\in [3/8;1/2]$.

In the regime $M_3^2 \infeg K_3 \infeg M_3^{-2}M_2^{6}$ we replace (\ref{estimation bilineaire hhl poids}) with
\begin{equation}\label{estimation bilineaire hhl poids 2}
p(m_1+m_2,n_1+n_2) \lesssim M_{min}^{-2}M_{max}^2p(m_1,n_1)+M_{min}^{-5/2}K_{max}^{1/2}
\end{equation}
Indeed, this follows from the same argument as for (\ref{estimation bilineaire lhh poids 2}). Proceeding then as in the previous case, we infer the final bound
\[\ln(M_{max})M_{max}^{3-2b}M_{min}^{-1/2}(K_1\et K_2)^{1/2}(K_1\ou K_2)^b\normL{2}{p\cdot f_1^{K_2}}\normL{2}{p\cdot f_2^{K_2}}\]
\end{proof}
The end of this section is devoted to the short-time bilinear estimate for the difference equation.
\begin{proposition}
There exists $\mu_2>0$ small enough such that for any $T\in]0;1]$ and $u\in \overline{\Fl}(T)$, $v\in \Fl(T)$, 
\begin{equation}\label{estimation bilineaire difference globale}
\norme{\overline{\Nl}(T)}{\drx(uv)}\lesssim T^{\mu_2}\norme{\overline{\Fl}(T)}{u}\norme{\Fl(T)}{v}
\end{equation}
\end{proposition}
\begin{proof}
Similarly to (\ref{estimation bilineaire cle}), now it suffices to prove
\begin{multline}\label{estimation bilineaire difference cle}
M_{max}^2M_3^{-1}\sum_{K_3\supeg M_3^2}K_3^{-1/2}\normL{2}{\mathbb{1}_{D_{M_3,\infeg K_3}} \cdot f_1^{K_1}\star f_2^{K_2}}\\ \lesssim M_2^{2}(K_1\et K_2)^b(K_1\ou K_2)^{1/2}\normL{2}{f_1^{K_1}}\normL{2}{p\cdot f_2^{K_2}}
\end{multline}
We proceed as above, except that now $u$ and $v$ do not play a symmetric role anymore, thus we have to separate three cases. (\ref{estimation bilineaire difference cle}) then follows directly from (\ref{estimation trilineaire synthese}) in the cases $High\times High\rightarrow Low$ and $Low\times High\rightarrow High$ and from (\ref{estimation dyadique poids}) in the case $High\times Low\rightarrow High$.
\end{proof}

\section{Proof of Theorem~\ref{theoreme}}\label{section preuve}
We finally turn to the proof of our main result. 

The starting point is the local well-posedness result for smooth data of Iòrio and Nunes.
\begin{theoreme}[\cite{IorioNunes}]\label{proposition existence haute regularite}
Assume $u_0\in \E^{\infty}$. Then there exists $T=T(\norme{\E^3}{u_0})\in]0;1]$ and a unique solution $u\in\mathcal{C}([-T;T],\E^{\infty})$ of (\ref{equation KP1-5}) on $[-T;T]\times\T^2$.
\end{theoreme}
\subsection{Global well-posedness for smooth data}
In view of this result and of the conservation of the energy, for Theorem~\ref{theoreme}~\ref{theoreme partie 1} it remains to prove (\ref{estimation energie avec donnee}), which will follow from the following proposition along with (\ref{estimation injection continue}).
\begin{proposition}\label{proposition controle solutions regulieres}
Let $\sigma\supeg 2$. For any $R>0$, there exists a positive $T=T(R)\sim R^{-1/(\mu_0\ou \mu_1)}$ such that for any $u_0\in\E^{\infty}$ with ${\displaystyle \norme{\E}{u_0}\infeg R}$, the corresponding solution $u\in\mathcal{C}([-T;T],\E^{\infty})$ satisfies
\begin{equation}\label{estimation controle norme avec donnee}
\norme{\Fl^{\sigma}(T)}{u}\infeg C_{\sigma} \norme{\E^{\sigma}}{u_0}
\end{equation}
\end{proposition}
\begin{proof}
We fix $\sigma\supeg 2$ and $R>0$ and take $u_0$ as in the proposition.

Let $T=T\left(\norme{\E^3}{u_0}\right)\in ]0;1]$ and $u\in\mathcal{C}([-T;T],\E^{\infty})$ be the solution to (\ref{equation KP1-5}) given by Theorem~\ref{proposition existence haute regularite}. Then, for $T'\in [0;T]$, we define
\begin{equation}\label{definition X argument continuite}
\X_{\sigma}(T') := \norme{\Bl^{\sigma}(T')}{u}+\norme{\Nl^{\sigma}(T')}{u\drx u}
\end{equation}
In order to perform our continuity argument, we will use the following lemma, whose proof is a straightforward adaptation of \cite[Lemma 8.3]{Article1}.
\begin{lemme}\label{lemme continuite}
Let $u\in\mathcal{C}\left([-T;T],\E^{\infty}\right)$, $\sigma\supeg 2$ and $T\in ]0;1]$. Then $\X_{\sigma} : [0;T]\rightarrow \R$ defined above is continuous and nondecreasing, and furthermore
\[\lim_{T'\rightarrow 0}\X_{\sigma}(T')\lesssim \norme{\E^{\sigma}}{u_0}\]
\end{lemme}
Recalling (\ref{equation estimation linéaire})-(\ref{estimation energie globale})-(\ref{estimation bilineaire globale}) for $\sigma \supeg 2$, we then get
\begin{equation}\label{estimations synthese}
\begin{cases}
{\displaystyle \norme{\Fl^{\sigma}(T)}{u}\lesssim \norme{\Bl^{\sigma}(T)}{u}+\norme{\Nl^{\sigma}(T)}{\drx(u^2)}}\\
{\displaystyle \norme{\Bl^{\sigma}(T)}{u}^2\lesssim \norme{\E^{\sigma}}{u_0}^2 + T^{\mu_0}\norme{\Fl(T)}{u}\norme{\Fl^{\sigma}(T)}{u}^2}\\
{\displaystyle \norme{\Nl^{\sigma}(T)}{\drx(uv)}\lesssim T^{\mu_1}\left(\norme{\Fl^{\sigma}(T)}{u}\norme{\Fl(T)}{v}+\norme{\Fl(T)}{u}\norme{\Fl^{\sigma}(T)}{v}\right)}
\end{cases}
\end{equation}
Thus, combining those estimates first with $\sigma=2$, we deduce that
\begin{equation}\label{estimation argument continuite}
\X_{2}(T)^2 \infeg c_1\norme{\E}{u_0}^2+c_2T^{\mu_0}\X_{2}(T)^3+c_3T^{2\mu_1}\X_{2}(T)^4
\end{equation}
for $T\in]0;1]$. Let us set $R:=c_1^{1/2}\norme{\E}{u_0}$. Then we choose $T_0=T_0(R)\in]0;1]$ small enough such that
\begin{equation*}
c_2T_0^{\mu_0}(2R)+c_3T_0^{2\mu_1}(2R)^2<1/2
\end{equation*}
Thus, using Lemma~\ref{lemme continuite} above and a continuity argument, we get that 
\[\X_{2}(T)\infeg 2R \text{ for } T\infeg T_0\]

Using then (\ref{equation estimation linéaire}), we deduce that
\begin{equation}\label{estimation contrôle norme uniforme}
\norme{\Fl(T)}{u} \lesssim \norme{\E}{u_0}
\end{equation}
for $T\infeg T_0$.

Using again (\ref{equation estimation linéaire})-(\ref{estimation energie globale})-(\ref{estimation bilineaire globale}) for $\sigma\supeg 3$ along with (\ref{estimation contrôle norme uniforme}), we then obtain
\begin{equation*}
\begin{cases}
{\displaystyle \norme{\Fl^{\sigma}(T)}{u}\lesssim \norme{\Bl^{\sigma}(T)}{u}+\norme{\Nl^{\sigma}(T)}{f}}\\
{\displaystyle \norme{\Bl^{\sigma}(T)}{u}^2\lesssim \norme{\E^{\sigma}}{u_0}^2 + T^{\mu_0}\norme{\E}{u_0}\norme{\Fl^{\sigma}(T)}{u}^2}\\
{\displaystyle \norme{\Nl^{\sigma}(T)}{\drx(u^2)}\lesssim T^{\mu_1}\norme{\E}{u_0}\norme{\Fl^{\sigma}(T)}{u}}
\end{cases}
\end{equation*}
We thus infer
\begin{equation*}
\X_{\sigma}(T)^2\infeg \widetilde{c_1}\norme{\E^{\sigma}}{u_0}^2+\widetilde{c_2}T^{\mu_0}R\X_{\sigma}(T)^2+\widetilde{c_3}T^{2\mu_1}R^2\X_{\sigma}(T)^2
\end{equation*}
So, up to choosing $T_0$ even smaller, such that
\begin{equation*}
\widetilde{c}_2T_0^{\mu_0}R+\widetilde{c_3}T_0^{2\mu_1}R^2 <1/2
\end{equation*}
we finally obtain (\ref{estimation controle norme avec donnee}).
\end{proof}

\subsection{Uniqueness}
Let $u$, $v$ be two global solutions of (\ref{equation KP1-5}) with data $u_0,v_0\in\E(\T^2)$, and fix $T_*>0$.

Using now (\ref{estimation lineaire difference})-(\ref{estimation energie difference L2})-(\ref{estimation bilineaire difference globale}), we get that for $T\in[0;T_*]$
\begin{equation}\label{estimations difference synthese}
\begin{cases}
{\displaystyle \norme{\overline{\Fl}(T)}{u-v}\lesssim \norme{\overline{\Bl}(T)}{u-v}+\norme{\overline{\Nl}(T)}{\drx\left((u-v)\frac{u+v}{2}\right)}}\\
{\displaystyle \norme{\overline{\Bl}(T)}{u-v}^2\lesssim \norme{L^2}{u_0-v_0}^2 + T^{\mu_0}\norme{\Fl(T)}{u+v}\norme{\overline{\Fl}(T)}{u-v}^2}\\
{\displaystyle \norme{\overline{\Nl}(T)}{\drx\left((u-v)\frac{u+v}{2}\right)}\lesssim T^{\mu_1}\norme{\Fl(T)}{u+v}\norme{\overline{\Fl}(T)}{u-v}}
\end{cases}
\end{equation}
Hence we infer
\begin{multline*}
\X_0(T)^2\lesssim \normL{2}{u_0-v_0}^2+T^{\mu_0}\norme{\Fl(T_*)}{u+v}\X_0(T)^2\\+T^{2\mu_1}\norme{\Fl(T_*)}{u+v}^2\X_0(T)^2
\end{multline*}
where, as in the previous subsection, \[\X_0(T):=\norme{\overline{\B}(T)}{u-v}+\norme{\overline{\Nl}(T)}{\drx\left((u-v)\frac{u+v}{2}\right)}\]
Thus, taking $T_0\in ]0;1[$ small enough such that
\[T_0^{\mu_0}\norme{\Fl(T_*)}{u+v}+T_0^{2\mu_1}\norme{\Fl(T_*)}{u+v}^2<1/2\]
we deduce that ${\displaystyle \X_0(T)\lesssim \normL{2}{u_0-v_0}}$ on $[0;T_0]$ which yields $u\equiv v$ on $[-T_0;T_0]$ provided $u_0=v_0$. Since $T_0$ only depends on ${\displaystyle \norme{\Fl(T_*)}{u+v}}$, we can then repeat this argument a finite number of time to reach $T_*$.

\subsection{Existence}
Again, in view of the conservation of mass, momentum and energy, it suffices to construct local in time solutions. To this aim we proceed as in \cite[Section 4]{IonescuKenigTataru2008}.\\

Take $R>0$, and let $u_0\in\E$ with ${\displaystyle \norme{\E}{u_0}\infeg R}$ and take ${\displaystyle (u_{0,j})\in\left(\E^{\infty}\right)^{\N}}$ with ${\displaystyle \norme{\E}{u_{0,j}}\infeg R}$, such that $(u_{0,j})$ converges to $u_0$ in $\E$. Using the same argument as for Theorem~\ref{theoreme}~\ref{theoreme partie 1}, it suffices to prove that there exists $T=T(R)>0$ such that ${\displaystyle \left(\Phi^{\infty}(u_{0,j})\right)}$ is a Cauchy sequence in $\mathcal{C}([-T;T],\E)$. Indeed, this provides the conservation of the mass, momentum energy for the corresponding limit, which allows us to extend the result to any time $T>0$.\\

Let $T=T(R)$ given by Proposition~\ref{proposition controle solutions regulieres}. For a fixed $M>1$ and $k,j\in\N$, we can split
\begin{multline*}
\norme{L^{\infty}_T\E}{\Phi^{\infty}(u_{0,k})-\Phi^{\infty}(u_{0,j})} \infeg \norme{L^{\infty}_T\E}{\Phi^{\infty}(u_{0,k})-\Phi^{\infty}(P_{\infeg M}u_{0,k})}\\+\norme{L^{\infty}_T\E}{\Phi^{\infty}(P_{\infeg M}u_{0,k})-\Phi^{\infty}(P_{\infeg M}u_{0,j})}+\norme{L^{\infty}_T\E}{\Phi^{\infty}(P_{\infeg M}u_{0,j})-\Phi^{\infty}(u_{0,j})}
\end{multline*}
The middle term is controled with the standard energy estimate
\begin{multline*}
\norme{L^{\infty}_T\E}{u-v}^2\\ \lesssim \norme{\E}{u_0-v_0}^2 + \left(\norme{L^1_TL^{\infty}}{u+v}+\norme{L^1_TL^{\infty}}{\drx(u+v)}\right)\norme{L^{\infty}_T\E}{u-v}^2\\
\lesssim \norme{\E}{u_0-v_0}^2 + T\norme{L^{\infty}_T\E^{10}}{u+v}\norme{L^{\infty}_T\E}{u-v}^2
\end{multline*}
 where the second line follows from a Sobolev inequality. Since
\begin{equation*}
\norme{L^{\infty}_T\E^{\sigma}}{\Phi^{\infty}(P_{\infeg M}u_{0,j})}\infeg C_{\sigma}\norme{\E^{\sigma}}{P_{\infeg M}u_{0,j}}
\end{equation*}
thanks to (\ref{estimation energie avec donnee}), we deduce that
\begin{equation*}
\norme{L^{\infty}_T\E}{\Phi^{\infty}(P_{\infeg M}u_{0,k})-\Phi^{\infty}(P_{\infeg M}u_{0,j})} \infeg C(M)\norme{\E}{u_{0,k}-u_{0,j}}
\end{equation*}
Therefore it remains to treat the first and last terms. Writing $u:=\Phi^{\infty}(u_{0,k})$, $v:=\Phi^{\infty}(P_{\infeg M}u_{0,k})$ and $w:=u-v$, a use of (\ref{estimation injection continue}) provides
\[\norme{L^{\infty}_T\E}{\Phi^{\infty}(u_{0,k})-\Phi^{\infty}(P_{\infeg M}u_{0,k})} \lesssim \norme{\Fl(T)}{w}\]
As before, defining now ${\displaystyle \widetilde{\X}(T'):=\norme{\B(T')}{w}+\norme{\Nl(T')}{w\drx w + v\drx w + w\drx v}}$, we get from (\ref{equation estimation linéaire})-(\ref{estimation energie difference globale})-(\ref{estimation bilineaire globale}) the bound
\[\widetilde{\X}(T')^2\lesssim \norme{\El}{u_0-v_0}^2 + T^{\mu_0}\norme{\Fl^3(T)}{v}\widetilde{\X}(T')^2+T^{2\mu_1}\norme{\Fl(T')}{u+v}^2\widetilde{\X}(T')^2\] 
From (\ref{estimation controle norme avec donnee}), we can bound 
\[\norme{\Fl^3(T)}{v}\lesssim \norme{\E^3}{P_{\infeg M}u_{0,k}}\lesssim C(M)\]
and
\[\norme{\Fl(T)}{u+v}^2\lesssim R^2\]
Thus, taking $M$ large enough and $T<T(R)$ small enough such that
\[T^{\mu_0}C(M) + T^{2\mu_1}R^2 <1/2\]
concludes the proof. 

\section*{Acknowledgements.}
The author is very grateful to Nikolay Tzvetkov for indicating this problem as well as for his valuable comments and remarks.

%
\bibliographystyle{plain}
\bibliography{biblio}

\begin{thebibliography}{10}

\bibitem{bourgain1993kdv}
Jean Bourgain.
\newblock Fourier transform restriction phenomena for certain lattice subsets
  and applications to nonlinear evolution equations.
\newblock {\em Geometric and Functional Analysis}, 3(3):209--262, 1993.

\bibitem{bourgain1993kp}
Jean Bourgain.
\newblock {On the Cauchy problem for the Kadomstev-Petviashvili equation}.
\newblock {\em Geometric and Functional Analysis}, 3(4):315--341, 1993.

\bibitem{Guo2017}
Boling Guo, Zhaohui Huo, and Shaomei Fang.
\newblock Low regularity for the fifth order kadomtsev–petviashvili-i type
  equation.
\newblock {\em Journal of Differential Equations}, pages~--, 2017.

\bibitem{GuoOh2015}
Zihua Guo and Tadahiro Oh.
\newblock {Non-Existence of Solutions for the Periodic Cubic {NLS} below
  ${L}^{{2}}$}.
\newblock {\em International Mathematics Research Notices}, 2016.

\bibitem{HadacHerrKoch}
Martin Hadac, Sebastian Herr, and Herbert Koch.
\newblock Well-posedness and scattering for the kp-ii equation in a critical
  space.
\newblock {\em Annales de l'Institut Henri Poincare (C) Non Linear Analysis},
  26(3):917 -- 941, 2009.

\bibitem{IonescuKenig2007}
A.~D. Ionescu and C.~E. Kenig.
\newblock {\em {Local and Global Wellposedness of Periodic KP-I Equations}},
  pages 181--212.
\newblock Princeton University Press, 2007.

\bibitem{IonescuKenigTataru2008}
A.D. Ionescu, C.E. Kenig, and D.~Tataru.
\newblock Global well-posedness of the {KP-I} initial-value problem in the
  energy space.
\newblock {\em Inventiones mathematicae}, 173(2):265--304, 2008.

\bibitem{IorioNunes}
Rafael~José Iório and Wagner Vieira~Leite Nunes.
\newblock On equations of {KP}-type.
\newblock {\em Proceedings of the Royal Society of Edinburgh: Section A
  Mathematics}, 128:725--743, 1 1998.

\bibitem{KP1970}
B.~B. {Kadomtsev} and V.~I. {Petviashvili}.
\newblock {On the Stability of Solitary Waves in Weakly Dispersing Media}.
\newblock {\em Soviet Physics Doklady}, 15, December 1970.

\bibitem{KenigPilod}
Carlos~E. Kenig and Didier Pilod.
\newblock Well-posedness for the fifth-order {K}d{V} equation in the energy
  space.
\newblock {\em Trans. Amer. Math. Soc.}, 367(4):2551--2612, 2015.

\bibitem{Koch2008}
H.~Koch and N.~Tzvetkov.
\newblock On finite energy solutions of the {KP-I} equation.
\newblock {\em Mathematische Zeitschrift}, 258(1):55--68, 2008.

\bibitem{KochTzvetkov2003BO}
Herbert Koch and Nikolay Tzvetkov.
\newblock {On the local well-posedness of the {B}enjamin-{O}no equation in
  ${H}^s(\mathbb{R})$}.
\newblock {\em International Mathematics Research Notices},
  2003(26):1449--1464, 2003.

\bibitem{LiXiao2008}
Junfeng Li and Jie Xiao.
\newblock Well-posedness of the fifth order kadomtsev–petviashvili i equation
  in anisotropic sobolev spaces with nonnegative indices.
\newblock {\em Journal de Mathématiques Pures et Appliquées}, 90(4):338 --
  352, 2008.

\bibitem{MST2002}
L.~Molinet, J.-C. Saut, and N.~Tzvetkov.
\newblock Well-posedness and ill-posedness results for the
  {K}adomtsev-{P}etviashvili-{I} equation.
\newblock {\em Duke Math. J.}, 115(2):353--384, 11 2002.

\bibitem{MST2011}
L.~Molinet, J.-C. Saut, and N.~Tzvetkov.
\newblock Global well-posedness for the {KP-II} equation on the background of a
  non-localized solution.
\newblock {\em Annales de l'Institut Henri Poincare (C) Non Linear Analysis},
  28(5):653 -- 676, 2011.

\bibitem{Article1}
T.~{Robert}.
\newblock {Global well-posedness of partially periodic KP-I equation in the
  energy space and application}.
\newblock {\em ArXiv e-prints}, June 2017.

\bibitem{SautTzvetkov1999}
J.-C. Saut and N.~Tzvetkov.
\newblock The cauchy problem for higher-order {KP} equations.
\newblock {\em Journal of Differential Equations}, 153(1):196 -- 222, 1999.

\bibitem{SautTzvetkov2000}
J.-C. Saut and N.~Tzvetkov.
\newblock {The Cauchy problem for the fifth order KP equations}.
\newblock {\em Journal de Mathématiques Pures et Appliquées}, 79(4):307 --
  338, 2000.

\bibitem{SautTzvetkov2001}
J.-C. {Saut} and N.~{Tzvetkov}.
\newblock {On Periodic {KP-I} Type Equations}.
\newblock {\em Communications in Mathematical Physics}, 221:451--476, 2001.

\bibitem{Tom}
Michael~M Tom.
\newblock On a generalized kadomtsev-petviashvili equation.
\newblock {\em Contemporary Mathematics}, 200:193--210, 1996.

\bibitem{Tzvetkov2004ill}
N.~{Tzvetkov}.
\newblock {Ill-posedness issues for nonlinear dispersive equations}.
\newblock {\em ArXiv Mathematics e-prints}, November 2004.

\bibitem{Zhang2015}
Yu~Zhang.
\newblock Local well-posedness of {KP-I} initial value problem on torus in the
  {B}esov space.
\newblock {\em Communications in Partial Differential Equations}, pages 1--26,
  2015.

\end{thebibliography}
\clearpage
\end{document}